\documentclass[a4paper,10pt]{amsart}
 
\usepackage[english]{babel}
\usepackage{dsfont} 
\usepackage{graphicx}
\usepackage{color}
\usepackage{hyperref}
\usepackage{cleveref}
\usepackage{enumerate}
\usepackage{amssymb} 
\usepackage{amsthm} 
\usepackage{xcolor}
\usepackage[foot]{amsaddr}

\allowdisplaybreaks

\theoremstyle{plain}
\newtheorem{thm}{Theorem}[section]
\newtheorem{lem}[thm]{Lemma}
\newtheorem{prop}[thm]{Proposition}
\newtheorem{cor}[thm]{Corollary}

\theoremstyle{definition}
\newtheorem{dfn}[thm]{Definition}
\newtheorem{bsp}[thm]{Example}
\newtheorem{exm}[thm]{Example}

\newtheorem{rem}[thm]{Remark}

\newcommand{\nset}{\mathds{N}}
\newcommand{\zset}{\mathds{Z}}
\newcommand{\qset}{\mathds{Q}}
\newcommand{\rset}{\mathds{R}}
\newcommand{\cset}{\mathds{C}}
\newcommand{\pset}{\mathds{P}}

\newcommand{\diag}{\mathrm{diag}\,}
\newcommand{\diff}{\mathrm{d}}
\newcommand{\lin}{\mathrm{lin}\,}

\newcommand{\pos}{\mathrm{Pos}}

\newcommand{\rank}{\mathrm{rank}\,}

\newcommand{\und}{\;\wedge\;}

\newcommand{\inter}{\mathrm{int}\,}

\newcommand{\cA}{\mathcal{A}}

\newcommand{\cat}{\mathcal{C}}

\newcommand{\cH}{\mathcal{H}}

\newcommand{\cS}{\mathcal{S}}

\newcommand{\cX}{\mathcal{X}}

\newcommand{\cV}{\mathcal{V}}

\newcommand{\cZ}{\mathcal{Z}}

\newcommand{\sA}{\mathsf{A}}
\newcommand{\sB}{\mathsf{B}}

\newcommand{\fA}{\mathfrak{A}}
\newcommand{\fM}{\mathfrak{M}}

\newcommand{\fc}{\mathfrak{c}}

\author[P.\ J.\ di~Dio]{Philipp J.\ di~Dio}
\address{Technische Universit\"at Berlin, Institut f\"ur Mathematik, Stra\ss{}e des 17.\ Juni 136, D-10623 Berlin, Germany}
\email{didio@tu-berlin.de}

\author[M. Kummer]{Mario Kummer}
\address{Technische Universit\"at Dresden, Fakult\"at Mathematik, Institut f\"ur Geometrie, Zellescher Weg 12-14, D-01062 Dresden, Germany}
\email{mario.kummer@tu-dresden.de}

\begin{title}[Carath\'eodory Numbers from Hilbert Functions]{The multidimensional truncated Moment Problem: Carath\'eodory Numbers from Hilbert Functions}
\end{title}

\begin{document}
\maketitle

\begin{abstract}
In this paper we improve the bounds for the Carath\'eodory number, especially on algebraic varieties and with small gaps (not all monomials are present). We provide explicit lower and upper bounds on algebraic varieties, $\rset^n$, and $[0,1]^n$. We also treat moment problems with small gaps. We find that for every $\varepsilon>0$ and $d\in\nset$ there is a $n\in\nset$ such that we can construct a moment functional $L:\rset[x_1,\dots,x_n]_{\leq d}\rightarrow\rset$ which needs at least $(1-\varepsilon)\cdot\left(\begin{smallmatrix} n+d\\ n\end{smallmatrix}\right)$ atoms $l_{x_i}$. Consequences and results for the Hankel matrix and flat extension are gained. We find that there are moment functionals $L:\rset[x_1,\dots,x_n]_{\leq 2d}\rightarrow\rset$ which need to be extended to the worst case degree $4d$, $\tilde{L}:\rset[x_1,\dots,x_n]_{\leq 4d}\rightarrow\rset$, in order to have a flat extension.
\end{abstract}


\noindent
\textbf{Key  words:} truncated moment problem, Carath\'eodory number, flat extension, generalized eigenvalues, algebraic variety, Hilbert function

\setcounter{tocdepth}{1}
\tableofcontents

\section{Introduction}

The theory of (truncated) moment sequences is a field of diverse applications and connections to numerous other mathematical fields, see e.g.\ \cite{stielt94,shohat43,akhiezClassical,kreinMarkovMomentProblem,kemper68,kemper87,landauMomAMSProc, marshallPosPoly,lauren09,fialkow10,lasserreSemiAlgOpt,schmudMomentBook}, and references therein. For more on recent advances in the reconstruction of measures from moments see e.g.\ \cite{dai92,milanf95,curto05,laurent05,gravin12,helton12,nie14,fialkoMomProbSurv,fialkow17,brehard19,didio19DerivMom}, and references therein.

A crucial fact in the theory of truncated moment sequences is the {Richter (Richter--Rogosinski--Rosenbloom)} Theorem \cite{richte57,rogosi58,rosenb52} which states that every truncated moment sequence is a convex combination of finitely many Dirac measures, see also \Cref{thm:richter}. The Carath\'eodory number is the minimal number $N$ such that every truncated moment sequence (with fixed truncation) is a sum of $N$ atoms, i.e., Dirac measures. It has been studied in several contexts but in most cases the precise value of the Carath\'eodory number is not known \cite{richte57,stroud71,moller76,reznick92,kunertPhD14,rienerOptima,didio17Cara,didioConeArXiv}.

In this work we proceed the study of Carath\'eodory numbers. We treat moment sequences with small gaps (see \Cref{sec:gaps}), moment sequences of measures supported on algebraic varieties (\Cref{sec:smoothcurves}), and the multidimensional polynomial case on $\rset^n$ and $[0,1]^n$ (\Cref{sec:cara}). For moment functionals with small gaps we find explicit lower and upper bounds for dimension $n=1$ based on Descartes' rule of signs, see \Cref{thm:sparsebounds}. For moment functionals $L:\rset[\cX]_{\leq 2d}\rightarrow\rset$ on polynomial functions on an algebraic set $\cX\subset\rset^n$ and for sufficiently large $d$, \Cref{thm:maincara} yields an upper bound of $P(2d)-1$ and a lower bound of
\[P(2d)-k\cdot P(d)+\binom k2,\]
where $P$ is the Hilbert polynomial and $k$ the dimension of $\cX$. In the case $\cX=\rset^n$ and $L:\rset[x_1,\dots,x_n]_{\leq 2d}\rightarrow\rset$, this gives the lower bound
\[\begin{pmatrix} n+2d\\ n\end{pmatrix} - n\cdot \begin{pmatrix} n+d\\ n\end{pmatrix} + \begin{pmatrix} n\\ 2\end{pmatrix}\]
(\Cref{thm:lowerboundsRn}). We obtain similar bounds for odd degrees and the case $\cX=[0,1]^n$ in \Cref{sec:cara}. In \Cref{sec:flat} we discuss implications of these bounds, when $n\rightarrow\infty$ and $d\rightarrow\infty$. We show that there are moment functionals $L:\rset[x_1,\dots,x_n]_{\leq 2d}\rightarrow\rset$ that behave as bad as possible under flat extensions, see \Cref{thm:flatExtensionBound} for the precise statement. For literature on flat extensions in this context see \cite{curto2,curto3,lauren09,schmudMomentBook} and the references therein.

\section{Preliminaries}

\subsection{Truncated Moment Problem} Let $\cA$ be a (finite dimensional) real vector space of measurable functions on a measurable space $(\cX,\fA)$. Denote by $L:\cA\rightarrow\rset$ a continuous linear functional. If there is a (positive) measure $\mu$ on $(\cX,\fA)$ such that
\begin{equation}\label{eq:momentFunctionalMeasure}
L(a) = \int_\cX a(x)~\diff\mu(x),\ \text{for all}\ a\in\cA,
\end{equation}
then $L$ is called a moment functional. If $\cA$ is finite dimensional, it is a truncated moment functional. By $\sA = \{a_1,\dots,a_m\}$ we denote a basis of the $m$-dimensional real vector space $\cA$ and by
\[s_i := L(a_i)\]
the $a_i$-th (or simply $i$-th) moment of $L$ (or $\mu$ for a $\mu$ as in (\ref{eq:momentFunctionalMeasure})). Given a sequence $s = (s_1,\dots,s_m)\in\rset^m$ we define the \emph{Riesz functional} $L_s$ by setting $L_s(a_i) = s_i$ for all $i=1,\dots,m$ and extending it linearly to $\cA$, i.e., the Riesz functional induces a bijection between moment sequences $s=(s_1,\dots,s_m)$ and moment functionals $L = L_s$. By $\fM_\sA$ we denote the set of all measures on $(\cX,\fA)$ such that all $a\in\cA$ are integrable and by $\fM_\sA(s)$ or $\fM_\sA(L)$ we denote all representing measures of the moment sequence $s$ resp.\ moment functional $L$. Even though moment sequences and moment functionals are the same, when we apply techniques from algebraic geometry it is easier to work with moment functionals $L:\cA\rightarrow\rset$ on e.g.\ $\cA = \rset[x_1,\dots,x_n]_{\leq 2d}$ or $\rset[\cX]_{\leq 2d}$ while when we work with Hankel matrices it is easier to work with moment sequences $s$ in a fixed basis $\sA$ of $\cA$. Since the polynomials $\rset[x_1,\dots,x_n]_{\leq 2d}$ are of special importance, we denote by
\[\sA_{n,d}: \{x^\alpha \,|\,\alpha\in\nset_0^n \und |\alpha|=\alpha_1 + \dots + \alpha_m \leq d\}\]
the monomial basis, where we have $x^\alpha = x_1^{\alpha_1}\cdots x_n^{\alpha_n}$ with $\alpha=(\alpha_1,\dots,\alpha_n)\in\nset_0^n$. On $\nset_0^n$ we work with the partial order $\alpha=(\alpha_1,\dots,\alpha_n) \leq \beta = (\beta_1,\dots,\beta_n)$ if $\alpha_i \leq \beta_i$ for all $i=1,\dots,n$.

\begin{dfn}
Let $\sA = \{a_1,\dots,a_m\}$ be a basis of the finite dimensional vector space $\cA$ of measurable functions on the measurable space $(\cX,\fA)$. We define $s_\sA$ by
\[s_\sA: \cX\rightarrow\rset^m,\quad x\mapsto s_\sA(x):= \begin{pmatrix}a_1(x)\\ \vdots\\ a_m(x)\end{pmatrix}.\]
\end{dfn}

Of course, $s_\sA(x)$ is the moment sequence of the Dirac $\delta_x$ measure and the corresponding moment functional is the point evaluation $l_x$ with $l_x(a) := a(x)$. By a measure we always mean a positive measure unless it is explicitly denoted as a signed measure.

The fundamental theorem in the theory of truncated moments is the following.

\begin{thm}[Richter Theorem {\cite[Satz 11]{richte57}}]\label{thm:richter}
Let $\sA = \{a_1,\dots,a_m\}$, $m\in\nset$, be finitely many measurable functions on a measurable space $(\cX,\fA)$. Then every moment sequence $s\in\cS_\sA$ has a $k$-atomic representing measure
\[s = \sum_{i=1}^k c_i\cdot s_\sA(x_i)\]
with $k\leq m$, $c_1,\dots,c_k>0$, and $x_1,\dots,x_k\in\cX$.
\end{thm}

The theorem can also be called \emph{Richter--Rogosinski--Rosenbloom Theorem} \cite{richte57,rogosi58,rosenb52}, see the discussion after Example 20 in \cite{didioConeArXiv} for more details. That every truncated moment sequence has a $k$-atomic representing measure ensures that the Carath\'eodory number $\cat_\sA$ is well-defined.

\begin{dfn}\label{dfn:caraNumber}
Let $\sA = \{a_1,\dots,a_m\}$ be linearly independent measurable functions on a measurable space $(\cX,\fA)$. For $s\in\cS_\sA$ we define the \emph{Carath\'eodory number $\cat_\sA(s)$ of $s$} by
\[\cat_\sA(s) := \min \{k\in\nset_0 \,|\, \exists\mu\in\fM_\sA(s)\ k\text{-atomic}\}.\]
We define the \emph{Carath\'eodory number $\cat_\sA$ of $\cS_\sA$} by
\[\cat_\sA := \max_{s\in\cS_\sA} \cat_\sA(s).\]The same definition holds for moment functionals $L:\cA\rightarrow\rset$.
\end{dfn}

The following theorem turns out to be a convenient tool for proving lower bounds on the Carath\'eodory number $\cat_\sA$.

\begin{thm}[{\cite[Thm.\ 18]{didio17Cara}}]\label{thm:caraLowerZeroSet}
Let $\sA = \{a_1,\dots,a_m\}$ be measurable functions on a measurable space $(\cX,\fA)$, $s\in\cS_\sA$, and $a\in\cA$ with $a\geq 0$ on $\cX$, $\cZ(a) = \{x_1,\dots,x_k\}$ and $L_s(a) = 0$. Then
\[\cat_\sA \quad\geq\quad \cat_\sA(s) \quad=\quad \dim\lin\{s_\sA(x_i) \,|\, i=1,\dots,k\}.\]
\end{thm}

\begin{rem}
Note that in \Cref{thm:caraLowerZeroSet} it is crucial that the zero set of $a$ is finite: Take $a=0$ and $\cX=\rset^n$ for a simple example where the statement fails when the zero set is not finite. 
\end{rem}

It is well-known that in general not every sequence $s\in\rset^m$ or linear functional $L:\cA\rightarrow\rset$ has a positive representing measure. But of course it always has a signed $k$-atomic representing measure with $k\leq m$.

\begin{lem}[{\cite[Prop.\ 12]{didioConeArXiv}}]\label{lem:signedMeasures}
Let $\sA = \{a_1,\dots,a_m\}$ be a basis of the finite dimensional space $\cA$ of measurable functions on a measurable space $(\cX,\fA)$. There exist points $x_1,\dots,x_m\in\cX$ such that every vector $s\in\rset^m$ has a signed $k$-atomic representing measure $\mu$ with $k\leq m$ and all atoms are from $\{x_1,\dots,x_m\}$,  i.e., every functional $L:\cA\rightarrow\rset$ is the linear combination $L = c_1 l_{x_1} + \cdots + c_m l_{x_m}$, $c_i\in\rset$.
\end{lem}

It is well-known that in dimension $n=1$ the atom positions $x_i$ of a moment sequence can be calculated from the generalized eigenvalue problem, see e.g.\ \cite{golub99}.
To formulate this and other results we introduce the following shift.

\begin{dfn}\label{dfn:multipl}
Let $n,d\in\nset$ and $s = (s_\alpha)_{\alpha\in\nset_0^n:|\alpha|\leq d}$. For $\beta\in\nset_0^n$ with $|\beta|\leq d$ we define $M_\beta s := (M_\beta s_\alpha)_{\alpha\in\nset_0^n:|\alpha+\beta|\leq d}$ by $M_\beta s_\alpha := s_{\alpha+\beta}$, i.e., $(M_\beta L)(p) = L(x^\beta\cdot p)$.
\end{dfn}

For a space $\cA$ of measurable functions with basis $\sA = \{a_1,a_2\dots\}$ the \emph{Hankel matrix} $\cH_d(L)$ of a linear functional $L:\cA^2\rightarrow\rset$ is given by $\cH_d(L) = (L(a_i a_j))_{i,j=1}^d$. The atom positions of a truncated moment sequence $s$ (resp.\ moment functional $L$) are then determined by the following result from a generalized eigenvalue problem.

\begin{lem}\label{lem:genEigenvalueOneDim}
Let $n,d\in\nset$, $\cX = \cset$, and $s = (s_0,s_1,\dots,s_{2d+1})\in\rset^{2d+2}$ with
\[s = \sum_{i=1}^k c_i\cdot s_{\sA_{1,2d+1}}(z_i)\]
for some $z_i\in\cset$, $c_i\in\cset$, and $k\leq d$. Then the $z_i$ are unique and are the eigenvalues of the generalized eigenvalue problem
\begin{equation}\label{eq:genEigValProb}
\cH_{d}(M_1 s)v_i = z_i \cH_{d}(s)v_i.
\end{equation}
\end{lem}
\begin{proof}
That the $z_i$ are the eigenvalues of (\ref{eq:genEigValProb}) and therefore uniqueness follows from
\[\cH_{d}(s) = (s_{\sA_{1,d}}(z_1),\dots,s_{\sA_{1,d}}(z_k))\cdot\diag(c_1,\dots,c_k)\cdot (s_{\sA_{1,d}}(z_1),\dots,s_{\sA_{1,d}}(z_k))^T\]
and
\begin{multline*}
\cH_{d}(M_1 s) =\\ (s_{\sA_{1,d}}(z_1),\dots,s_{\sA_{1,d}}(z_k))\cdot\diag(c_1 z_1,\dots,c_k z_k)\cdot (s_{\sA_{1,d}}(z_1),\dots,s_{\sA_{1,d}}(z_k))^T.\qedhere
\end{multline*}
\end{proof}

We gave here only the $1$-dimensional formulation, but a similar result holds also for $n>1$. But as seen from the Carath\'eodory number and the flat extension in \Cref{sec:cara} and \Cref{sec:flat}, the size of the Hankel matrix of the flat extension can be very large. For numerical reasons it is therefore advisable to reduce $n$-dimensional problems to $1$-dimensional problems.

\subsection{Algebraic Geometry}
Consider the polynomial ring $\rset[x_0,\ldots,x_n]$ with the natural grading and let $I\subset\rset[x_0,\ldots,x_n]$ be a homogeneous ideal. Let
\[R=\rset[x_0,\ldots,x_n]/I\]
be the quotient ring which is a graded ring itself. Recall that the \emph{Hilbert function} of $R$ is given by $HF_R(d)=\dim R_d$ where $R_d$ is the degree $d$ part of $R$. For $d$ large enough one has $HF_R(d)=HP_R(d)$ for some polynomial $HP_R$ of degree $k$ which is called the \textit{Hilbert polynomial} of $R$. 

In this article, we will always denote by $\pset^n=\pset_{\cset}^n$ the complex projective space. A \emph{real projective variety} is the zero set $V\subset\pset^n$ of some homogeneous ideal $I\subset\rset[x_0,\ldots,x_n]$. In particular, a real projective variety can contain nonreal points but it is defined by real polynomial equations. We will denote by $V(\rset)$ the set of real points of $V$. The Zariski closure of any subset $W\subset\pset^n$, that consists only of real points, is an example for a real projective variety $V$ with the additional property that $V(\rset)$ is Zariski in $V$.
If $V\subset\pset^n$ is a real projective variety and $I$ is its homogeneous vanishing ideal, then the Hilbert function/polynomial $HF_V$ resp. $HP_V$ of $V$ is the Hilbert function/polynomial of $\rset[x_0,\ldots,x_n]/I$. In this case, the leading coefficient of $HP_V$ is $\frac{e}{k!}$ where $e$ is the degree of $V$.

Now we consider the dehomogenization map
\[\rset[x_0,\ldots,x_n]\to\rset[x_1,\ldots,x_n],\quad f\mapsto f|_{x_0=1}.\]
Let $I\subset\rset[x_1,\ldots,x_n]$ be an ideal and $I^h\subset\rset[x_0,\ldots,x_n]$ the \emph{homogenization} of $I$, i.e., the ideal generated by the homogenizations $f^h$ of all $f\in I$. Then the dehomogenization map induces an isomorphism of vector spaces
\[(\rset[x_0,\ldots,x_n]/I^h)_d \quad\to\quad (\rset[x_1,\ldots,x_n]/I)_{\leq d}\]
for all $d\geq0$. Here $(\rset[x_1,\ldots,x_n]/I)_{\leq d}$ is the subspace of $\rset[x_1,\ldots,x_n]/I$ consisting of the residue classes of polynomials of degree at most $d$. The main application of this observation will be the case when $I$ is the vanishing ideal of finitely many points $\Gamma$ in $\rset^n$. In this case the dimension
\[\dim\lin \{s_{\sA_{n,d}}(x) \,|\, x\in\Gamma\}\]
of the span of the point evaluations $s_{\sA_{n,d}}(x)$ in $\rset[x_1,\ldots,x_n]_{\leq d}^*$ at the points from $\Gamma$ {needed \Cref{thm:caraLowerZeroSet}} is
\[\dim (\rset[x_1,\ldots,x_n]/I)_{\leq d} \quad=\quad \dim (\rset[x_0,\ldots,x_n]/I^h)_d \quad=\quad HF_I(d).\]
The Hilbert function $HF_I$ of an ideal $I$ can be easily calculated if it is generated by a regular sequence.

\begin{dfn}
Let $A$ be a commutative ring. A sequence $f_1,\ldots,f_r\in A$ is a \emph{regular sequence} if for all $i=1,\ldots,r$ the residue class of $f_i$ is not a zero divisor in $A/(f_1,\ldots,f_{i-1})$.
\end{dfn}

The following is a consequence of Krull's Principal Ideal Theorem. We include a proof since we are not aware of a good reference.

\begin{lem}\label{lem:krull}
Let $I\subset\rset[x_0,\ldots,x_n]$ be a homogeneous radical ideal and $V\subset\pset^n$ its zero set. If each irreducible component of $V$ has the same dimension $d\geq1$, then for any homogeneous $f\in\rset[x_0,\ldots,x_n]$ the following are equivalent:
\begin{enumerate}[i)]
\item $f$ is not a zero divisor in $\rset[x_0,\ldots,x_n]/I$.
\item $f$ is not in a minimal prime ideal of $\rset[x_0,\ldots,x_n]/I$.
\item $f$ is not identically zero on an irreducible component of $V$.
\item Each irreducible component of $V\cap\cV(f)$ has dimension at most $d-1$.
\item Each irreducible component of $V\cap\cV(f)$ has dimension $d-1$.
\end{enumerate}
Furthermore, if $f$ is not constant and $V$ nonempty, then $V\cap\cV(f)$ is nonempty.
\end{lem}
\begin{proof}
The minimal prime ideals of the homogeneous coordinate ring $$A=\rset[x_0,\ldots,x_n]/I$$ of $V$ are exactly the vanishing ideals of irreducible components of $V$. Thus we have $(ii)\Leftrightarrow(iii)$. If $f$ is a zero divisor in $A$, then there is a nonzero $g\in A$ such that $fg=0$. Let $V_i$ be an irreducible component of $V$ on which $g$ does not vanish identically. Then $V_i\subset\cV(f)\cup\cV(g)$ implies $V_i\subset\cV(f)$ because $V_i$ is irreducible. Thus $(iii)$ implies $(i)$.  By \cite[p.~44, Ex.~9]{atiyahIntroComAlg} every minimal prime ideal contains only zero divisors. This shows $(i)\Rightarrow(ii)$. If $f$ vanishes entirely on an irreducible component $V_i$ of $V$, then $V_i$ is an irreducible component of $V\cap\cV(f)$. By assumption we have $\dim(V_i)=d$, so we cannot have $(iv)$. Thus $(iv)\Rightarrow(iii)$.
 
Since $(v)$ clearly implies $(iv)$, it remains to show $(v)$ under the assumption of $(i)-(iii)$. If $f$ is a unit in $A$, then $V\cap\cV(f)=\emptyset$ and $(v)$ is trivially true as there are no irreducible components. Thus we can assume that $f$ is neither a zero divisor nor a unit in $A$. Thus by Krull's Principal Ideal Theorem \cite[Cor.~11.17]{atiyahIntroComAlg} every minimal prime ideal over $(f)\subset A$ has height one. This implies that every irreducible component $W$ of $V\cap\cV(f)$ has codimension one. Since $V$ is of pure dimension $d$, this means that the dimension of $W$ is $d-1$. The additional statement follows for example by \cite[Cor.~1.7]{shafa77} because we have $\dim(V)>0$.
\end{proof}

\begin{cor}\label{cor:regseq}
Let $I_0\subset\rset[x_0,\ldots,x_n]$ be a homogeneous prime ideal such that $\dim\cV(I_0)\geq k$. Let $f_1,\ldots,f_k\in\rset[x_0,\ldots,x_n]$ homogeneous elements of positive degree such that for all $i=1,\ldots,k$ we have:
\begin{enumerate}[i)]
\item $I_i:=I+(f_1,\ldots,f_i)$ is radical.
\item $\dim\cV(I_i)=\dim\cV(I_{i-1})-1$.
\end{enumerate}
Then $f_1,\ldots,f_k$ is a regular sequence modulo $I_0$.
\end{cor}
\begin{proof}
For $i=0,\ldots,k$ let $V_i=\cV(I_i)\subset\pset^n$ and let $d=\dim(V_0)$. First we show that each irreducible component of $V_i$ has dimension $d-i$ by induction on $i$. The claim is clear for $i=0$ because $I$ is a prime ideal. Assume the claim is true for $0\leq i<k$. Then we can apply \Cref{lem:krull} to the ideal $I_i$. By assumption we have $\dim\cV(I_{i+1})=\dim\cV(I_{i})-1=d-i-1$ so we have $(iv)$. Thus we also have $(v)$ which says that each irreducible component of $V_{i+1}$ has dimension $d-i-1$. Then by the same lemma we also have that $f_{i+1}$ is not a zero divisor modulo $I_i$ which shows that $f_1,\ldots,f_k$ is a regular sequence modulo $I_0$.
\end{proof}

\begin{lem}\label{lem:hilfun}
Let $I\subset\rset[x_0,\ldots,x_n]$ be a homogeneous ideal and $R=\rset[x_0,\ldots,x_n]/I$ with Hilbert function $HF_R$. Let $f_1,\ldots, f_r\in R$ be a regular sequence of homogeneous elements of degree $d$. The Hilbert function $HF_{R/(f_1,\ldots,f_r)}$ of $R/(f_1,\ldots,f_r)$ is
\[HF_{R/(f_1,\ldots,f_r)}(j) \quad=\quad\sum_{i=0}^r (-1)^i\cdot \binom ri\cdot HF_{R}(j-id).\]
\end{lem}
\begin{proof}
We prove the statement by induction on $r$. The case $r=0$ is trivial. In order to prove the induction step, let $R^i = R/(f_1,\ldots,f_i)$ for $i=0,\ldots,r$. For all $j\in\zset$ we have the exact sequence
\[0\to R^{r-1}_{j-d}\to R^{r-1}_{j}\to R^r_j\to 0\]
where the first map is given by multiplication with $f_r$. Therefore
\[HF_{R^r}(j)=HF_{R^{r-1}}(j)-HF_{R^{r-1}}(j-d).\]
By induction hypothesis this implies that
\begin{align*}
HF_{R^{r}}(j) &= \sum_{i=0}^{r-1}(-1)^i\binom {r-1}i HF_{R}(j-i\cdot d)-\sum_{i=0}^{r-1}(-1)^i\binom {r-1}i HF_{R}(j-(i+1)d)\\
&=\sum_{i=0}^{r}(-1)^i(\binom {r-1}i + \binom{r-1}{i-1})HF_{R}(j-i\cdot d)\\
&=\sum_{i=0}^r(-1)^i\binom ri HF_{R}(j-i\cdot d).\qedhere
\end{align*}
\end{proof}

At various places we will make use of the following version of Bertini's Theorem. 

\begin{thm}
Let $\cX\subset\pset^n$ be a real projective variety of dimension $k$. Then the following statements hold for \emph{generic} homogeneous forms $f_1,\ldots,f_r\in\rset[x_0,\ldots,x_n]$, $r\leq k$, of degree $d>0$ in the sense that the set of exceptions is contained in a lower dimensional algebraic subset of $\rset[x_0,\ldots,x_n]_d^r$. 
\begin{enumerate}[i)]
\item The homogeneous vanishing ideal of $\cX\cap\cV(f_1,\ldots,f_r)$ is generated by the homogeneous vanishing ideal of $\cX$ and $f_1,\ldots,f_r$.
\item If $\cX$ is irreducible and $r<k$, then $\cX\cap\cV(f_1,\ldots,f_r)$ is irreducible as well.
\item We have $\dim(\cX\cap\cV(f_1,\ldots,f_r))=k-r$.
\item If the singular locus of $\cX$ has dimension at most $r-1$, then $\cX\cap\cV(f_1,\ldots,f_r)$ is smooth.
\end{enumerate}
\end{thm}
\begin{proof}
Bertini's Theorem in its usual formulation says that the above listed statements hold for generic homogeneous forms $f_1,\ldots,f_r\in\cset[x_0,\ldots,x_n]$, $r\leq k$, of degree $d>0$. As a reference for this see for example \cite[Thm. 6.10, Cor. 6.11]{jouano83}. This means that the set $U$ of exceptions is contained in a lower dimensional algebraic subset $W\subset\cset[x_0,\ldots,x_n]_d^r$. The set $U'$ of tuples $(f_1,\ldots,f_r)\in \rset[x_0,\ldots,x_n]_d^r$ of real polynomials for which one of our statements does not hold is thus contained in the algebraic subset $W'=W\cap \rset[x_0,\ldots,x_n]_d^r$ of $\rset[x_0,\ldots,x_n]_d^r$. Since the set of real points $\rset[x_0,\ldots,x_n]_d^r$ is Zariski dense in the vector space $\cset[x_0,\ldots,x_n]_d^r$, we see that $W$ does not contain $\rset[x_0,\ldots,x_n]_d^r$. Thus $W'$ is a strict algebraic subset of $\rset[x_0,\ldots,x_n]_d^r$. This shows the claim.
\end{proof}

For more on Hilbert functions and polynomials see e.g.\ \cite{stanley78}, or standard text books on commutative algebra like \cite{eisenComAlg}, \cite{eisenGeomSyz}, or \cite{BrunsCohenMacaulay}.

\section{Carath\'eodory Numbers for Moment Sequences with small Gaps}
\label{sec:gaps}

We want to start our investigation of the Carath\'eodory number in the 1-dimen\-sional case with gaps, i.e., not all monomials are present.

Let $d_1,\dots,d_r\in\nset$ be some natural numbers whose greatest common divisor is one. We consider the subring $R=\rset[t^{d_1},\ldots,t^{d_r}]$ of $\rset[t]$. By $R_{\leq d}$ we denote the vector space of polynomials in $R$ of degree at most $d$. By the assumption on the greatest common divisor there is a constant $c$ such that $t^d\in R$ for all $d\geq c$. We choose $c$ minimal with this property and denote it by $\fc$. We observe that one has
\[\dim R_{\leq d}=d+1-g \quad\text{for}\quad d\geq \fc\]
where $\fc+1-{g}$ is the number of monomials in $R$ of degree at most $\fc$. In other words, $g$ is the number of monomials that are not in $R$ (i.e., the number of gaps).

\begin{dfn}
The $k$-th \emph{Descartes number} $D_k$ of $R$ is the maximal number of different real zeros that a polynomial $f\in R_{\leq k}$ can have.
\end{dfn}

Recall that Descartes' rule of signs says that the number of positive real zeros (counted with multiplicities) of a polynomial $f=\sum_{k=0}^n c_k t^k$ is bounded from above by the number $\textrm{Var}(c_0,\ldots,c_n)$ of sign changes in the sequence $c_0,\ldots,c_n$ after erasing all zeros. The number of negative zeros (again counted with multiplicities) of $f$ is then bounded by $\textrm{Var}(c_0,-c_1,\ldots,(-1)^nc_n)$. Conversely, Grabiner \cite{grab99} constructed for all sequences of signs $(\sigma_0,\ldots,\sigma_n)$, $\sigma_i\in\{0,\pm 1\}$, a polynomial $f=\sum_{k=0}^n c_k t^k$ with only simple positive and negative zeros and $\textrm{sgn}(c_i)=\sigma_i$, that realizes both bounds. Thus Descartes' rule of signs gives a purely combinatorial way to determine an upper bound on the $k$-th Descartes number from the numbers $d_1,\ldots,d_r$. This also shows that $D_k$ is the maximal number of different real zeros that a polynomial $f\in R_{\leq k}$ with $f(0)\neq0$ can have since adding a small constant of appropriate sign does not decrease the number of real zeros of a polynomial whose only possibly multiple real root is $0$.

\begin{bsp}\label{exp:somesparse}\
\begin{enumerate}[a)]
\item Let $R=\rset[t^4,t^6,t^7]$. Then the Descartes number $D_7$ is the maximal number of real roots that a polynomial of the form $a+bt^4+ct^6+dt^7$ can have. By trying out all possible signs on the coefficients, we find by Descartes' rule of signs that such a polynomial can have at most five real zeros and by \cite{grab99} there actually is such a polynomial. Thus $D_7=5$.

\item The Descartes number does not only depend on the number of involved monomials but also on their parities. For example if $R=\rset[t^5,t^6,t^9]$, then $D_9=3$.
\end{enumerate}
\end{bsp}

\begin{prop}
For all $k\geq 0$ we have $D_{\fc+k}=D_{\fc}+k$.
\end{prop}
\begin{proof}
We prove the claim by induction on $k$. The case $k=0$ is trivial. Let $k\geq1$ and assume that the claim is true for $k-1$. Then there is a sequence of signs $(\sigma_0,\ldots,\sigma_{\fc+k-1})$, $\sigma_i\in\{0,\pm 1\}$, $\sigma_0\neq0$, of coefficients of a polynomial in $R$ with $D_{\fc}+k-1$ different real zeros. In particular,
\[\textrm{Var}(\sigma_0,\ldots,\sigma_{\fc+k-1})+\textrm{Var}(\sigma_0,\ldots,(-1)^{\fc+k-1}\sigma_{\fc+k-1})=D_{\fc}+k-1.\]
Letting $\sigma_{\fc+k}=-\sigma_{\fc+k-1}$ we get that
\[\textrm{Var}(\sigma_0,\dots,\sigma_{\fc+k-1},\sigma_{\fc+k})+\textrm{Var}(\sigma_0,\ldots,(-1)^{\fc+k-1}\sigma_{\fc+k-1},(-1)^{\fc+k}\sigma_{\fc+k})=D_\fc+k\]
and another choice of $\sigma_{\fc+k}$ would not result in a larger sum, so $D_{\fc+k}=D_\fc+k$.
\end{proof}

\begin{prop}\label{prop:realzerossparse}
Let $k\geq \fc$. The maximal number of real zeros that a nonnegative polynomial $f\in R_{\leq 2k}$ can have is between $k-(\fc-D_\fc)$ and $k-\left \lceil{\frac{\fc-D_\fc-1}{2}}\right \rceil$. Here the lower bound is realized by a polynomial that is the square of an element of $R$.
\end{prop}
\begin{proof}
For the lower bound just take the square of a polynomial of degree $k$ with $D_k=D_\fc+k-\fc$ real zeros. On the other hand, if $f\in R$ is a nonnegative polynomial with $N$ real zeros, then $tf'\in R$ has at least $2N-1$ zeros. Therefore, $2N-1\leq D_{2k}$ implies
\[N\leq \left \lfloor{\frac{D_{2k}+1}{2}} \right \rfloor=\left \lfloor{\frac{D_{\fc}+2k-\fc+1}{2}} \right \rfloor=k-\left \lceil{\frac{\fc-D_{\fc}-1}{2}} \right \rceil.\qedhere\]
\end{proof}

\begin{lem}\label{lem:indepsparse}
The point evaluations $l_{p_1},\ldots,l_{p_n}: R_{\leq e}\to\rset$ are linearly independent for any pairwise distinct points $p_1,\ldots,p_n\in\rset$ and $e\geq \fc+n-1$.
\end{lem}
\begin{proof}
We consider the map $\psi:R_{\leq e}\to \rset^n,\, g\mapsto (g(p_i))_{1\leq i\leq n}$. The polynomial $t^\fc\prod_{i=1,i\neq j}^n(t-p_i)$ is mapped to a nonzero multiple of the $j$-th unit vector except for the case when $p_j=0$. Thus we have at least all unit vectors but one in the image and the constant polynomial $1$ is mapped to the vector $(1,\dots,1)$. Thus $\psi$ is surjective which implies the claim.
\end{proof}

The following lemma generalizes \cite[Thm.\ 3.68]{didioDiss} and \cite[Thm.\ 45]{didio17Cara}.

\begin{lem}\label{lem:upperGabs}
Let $\cA\subset\rset[x]$ be the vector space of polynomials on $\rset$ generated by the monomials $\sA = \{x^{d_1}, x^{d_2},\dots, x^{d_m}\}$, $m,d_i\in\nset$, such that $d_1 = 0 < d_2 < \dots< d_m$ and $d_m$ is even. If all non-negative polynomials in $\cA$ have at most $C$ zeros. Then
\[\cat_\sA \leq C + 1.\]
\end{lem}
\begin{proof}
Let $s\in\cS_\sA$ be a moment sequence.

Step i):  If $s$ is in the boundary of the moment cone there exists a $p\in\cA$ with $p\geq 0$ and $L_s(p)=0$, i.e., all point evaluations are located at the zeros of $p$. Hence, $s$ requires at most $C$ point evaluations.

Step ii): Assume now $s$ is in the interior of the moment cone.

Homogenize $\sA$, i.e., $\sB := \{y^{d_m},x^{d_2} y^{d_m-d_2},\dots,x^{d_m}\}$. Since $s$ is a moment sequence, we have $s = \sum_{i=1}^l c_i\cdot s_\sA(x_i) = \sum_{i=1}^l c_i\cdot s_\sB((x_i,1))$. Since $x^{d_m},y^{d_m}\in\sB$ we have $x^{d_m}+y^{d_m}>0$ on $\pset^1$ and the moment cone $\cS_\sB$ is closed. Hence, by \cite[Prop.\ 8]{didio17Cara} there exists an $\varepsilon>0$ such that
\[c_{q}(x,y) := \sup\{r\geq 0 \,|\, q - r\cdot s_\sB(x,y)\in \cS_\sB\} \tag{$*$}\]
is attained and continuous for all $q\in B_\varepsilon(s)$, $(x,y)\in\pset^1$, we have $B_\varepsilon(s)\subset\inter\cS$, and
\[\Gamma:=\sup_{q\in B_\varepsilon(s)} c_{q}((0,1))<\infty.\tag{$**$}\]
Let
\[T := \bigcup_{c\in [0,\Gamma+1]} \overline{B_\varepsilon(s-c\cdot s_\sB(0,1))}\]
be the $\varepsilon$-tube around the line $s - [0,\Gamma+1]\cdot s_\sB((0,1))$. Write $T = T_1 \cup T_2 \cup T_3$ with $T_1 := T\cap \inter\cS_\sB$, $T_2 := T\cap\partial\cS_\sB$, and $T_3 := T\setminus (T_1\cup T_2)$. I.e., $T_1$ is the part of the $\varepsilon$-tube inside the moment cone, $T_2$ is the intersection of the $\varepsilon$-tube with the boundary of the moment cone, and $T_3$ is the part of the $\varepsilon$-tube outside the moment cone.

Since the moment cone is closed (and convex), also $T_2$ is closed and every path starting in $T_1$ and ending in $T_3$ contains at least one point in $T_2$. We define
\[t' := t - c_q((0,1))\cdot s_\sB((0,1))\]
for all $t\in B_\varepsilon(s)$. By ($*$) and ($**$) we have for all $t\in B_\varepsilon(s)$ that $t'$ is a moment sequence without an atom at $(0,1)$ (by maximality of $c_q((0,1))$), i.e., $t'$ is a moment sequence on $\rset$ and in the boundary of $\cS_\sA$. By step (i) $t'$ requires at most $k$ point evaluations. Since $s_\sB((x,y))$ is continuous in $(x,y)$, there exists a $\delta = \delta(\varepsilon)>0$ such that
\[s' := s - c_s((\delta,1))\cdot s_\sB((\delta,1))\in T_2,\]
i.e., also $s'$ is a moment sequence on $\rset$ with at most $k$ point evaluations. Hence, $s = s' + c_s((\delta,1))\cdot s_\sB((\delta,1))$ is a moment sequence on $\rset$ with $\cat_\sA(s) \leq C+1$.

Finally, since $s\in\cS_\sA$ was arbitrary we have $\cat_\sA \leq C+1$.
\end{proof}

\begin{thm}\label{thm:sparsebounds}
Let $R=\rset[t^{d_1},\dots,t^{d_r}]$ and $k\geq \fc$. Every moment functional $L:R_{\leq 2k}\to\rset$ is a conic combination of at most $k+1-\left \lceil{\frac{\fc-D_\fc-1}{2}}\right \rceil$ point evaluations. There are moment functionals $L:R_{\leq 2k}\to\rset$ that are not a conic combination of less than $k-(\fc-D_\fc)$ point evaluations.
\end{thm}
\begin{proof}
For the upper bound we combine \Cref{prop:realzerossparse} and \Cref{lem:upperGabs}. We have $1\in R_{\leq 2k}$ and since $k\geq \fc$ we have by the minimality of $\fc$ (see second paragraph at the beginning of this section) that $x^{2k}\in R_{\leq 2k}$. So the monomial basis $\sA$ of $R_{\leq 2k}$ fulfills the conditions in \Cref{lem:upperGabs} and by \Cref{prop:realzerossparse} every non-negative polynomial in $R_{\leq 2k}$ has at most $C = k-\left \lceil{\frac{\fc-D_\fc-1}{2}}\right \rceil$ zeros. \Cref{lem:upperGabs} implies that every moment sequence/functional is represented by at most $C + 1 = k+1-\left \lceil{\frac{\fc-D_\fc-1}{2}}\right \rceil$ point evaluations.

The lower bound follows from \Cref{prop:realzerossparse}, \Cref{lem:indepsparse}, and \Cref{thm:caraLowerZeroSet}.
\end{proof}

\begin{bsp}\
\begin{enumerate}[a)]
\item Let $R=\rset[t^2,t^{2r+1}]$ with $r\geq0$. In this case we have $\fc=D_\fc=2r+1$. Thus for $k\geq 2r+1$ every moment functional $L:R_{\leq 2k}\to\rset$ is a conic combination of at most $k+1$ point evaluations and there are moment functionals which are not a conic combination of less than $k$ point evaluations.

\item Let $R=\rset[t^r,t^{r+1},t^{r+2},\ldots]$. Then $\fc=r$ and $D_\fc=1$ if $r$ is odd and $D_\fc=2$ if $r$ is even, so the difference between upper and lower bound in \Cref{thm:sparsebounds} grows linearly in $r$. This situation is in sharp contrast to the results from \Cref{sec:smoothcurves} on smooth curves.
\end{enumerate}
\end{bsp}

\section{Carath\'eodory Numbers for Measures supported on algebraic Varieties}
\label{sec:smoothcurves}

Now for any subset $\cX\subset\rset^n$ we are interested in the ring $\rset[\cX]$ of polynomial functions $\cX\to\rset$. The finite dimensional vector space of all functions $\cX\to\rset$ that can be represented by a polynomial of degree at most $d$ is denoted by $\rset[\cX]_{\leq d}$. If $I\subset\rset[x_1,\ldots,x_n]$ is the ideal of all polynomials vanishing on $\cX$, then $\rset[\cX]=\rset[V_0]=\rset[x_1,\ldots,x_n]/I$ where $V_0\subset\rset^n$ is the Zariski closure of $\cX$. Let $V\subset\pset^n$ be the Zariski closure of $V_0$ in the complex projective space. Then one has
\begin{equation}\label{eq:hfdim}
HF_V(d)\quad=\quad\dim \rset[\cX]_{\leq d}.
\end{equation}
From Richter's Theorem we thus immediately get the following.

\begin{prop}\label{prop:upper}
Every moment functional $L:\rset[\cX]_{\leq 2d}\to\rset$ is a conic combination of at most $HF_V(2d)$ point evaluations $l_{x_i}$ with $x_i\in\cX$. If $\cX$ consists of less than $HF_V(2d)$ path-connected components, then $HF_V(2d)-1$ point evaluations are sufficient. In particular, for large $d$ this upper bound grows like a polynomial whose degree is the dimension of the Zariski closure of $\cX$.
\end{prop}
\begin{proof}
In (\ref{eq:hfdim}) we already established that $\dim \rset[\cX]_{\leq 2d} = HF_V(2d)$. Since $\cX$ is a measurable space and monomials are measurable functions, Richter's \Cref{thm:richter} implies that $L$ can be represented by at most $\dim \rset[\cX]_{\leq 2d} = HF_V(2d)$ point evaluations.

Since $\cX$ is a topological space which consists of at most $HF_V(2d) -1$ path-connected components, $1\in \rset[\cX]$, and $\rset[\cX]$ consists of continuous functions we have that $s_\sA(\cX)$ consists of at most $HF_V(2d)-1$ path-connected components. All conditions of \cite[Thm.\ 12]{didio18gaussian} are fulfilled which implies the upper bound.
\end{proof}

In order to provide lower bounds as well, we will need the following lemma.

\begin{lem}\label{lem:totreal}
Assume that $V$ is irreducible with homogeneous vanishing ideal $I$ and that its singular locus has codimension at least $2$. If $k=\dim(V)$, then, for all $d$ large enough, there are $k$ real homogeneous polynomials $f_1,\ldots,f_k$ of degree $d$ whose common zero set $Z$ on $V$ consists of $d^k\cdot \deg(V)$ different points that are all real and contained in $V_0$. Furthermore, one can choose the $f_1,\ldots,f_k$ to be a regular sequence with the property that they generate the homogeneous vanishing ideal of $Z$ modulo $I$.
\end{lem}
\begin{proof}
By Bertini's theorem, for a generic choice of $k-1$ real linear forms $l_1,\ldots,l_{k-1}$ the set $V\cap\cV(l_1,\cdots,l_{k-1})\subset\pset^n$ is a real smooth irreducible curve $X$. Since the real points of $V$ are Zariski dense in $V$ we can furthermore assume that $X(\rset)$ is nonempty. Now by \cite[Cor. 2.10, Rem. 2.14]{scheider00}, for large enough $d$, there is a homogeneous polynomial $f$ of degree $d$ all of whose zeros on $X$ are real, simple and do not lie at the hyperplane at infinity. Since $\deg X=\deg V=:e$, these are $de$ many points. The same is true for the zeros of $f$ on $X'$ where $X'$ is the intersection of $V$ with linear forms $l_1',\dots,l_{k-1}'$ that are sufficiently small perturbations of $l_1,\dots,l_{k-1}$. Therefore, for sufficiently small $\epsilon>0$, the common zero set on $V$ of $f$ with the polynomials $f_i=\prod_{j=1}^d(l_i+j\epsilon\cdot x_0)$, $i=1,\ldots,k-1$, consists of exactly $d^ke$ real, simple points that do not lie in the hyperplane at infinity. Thus these $d^ke$ points lie in $V_0$.
 
In order to obtain the additional properties, we can perturb $f_1,\ldots,f_k$ a little bit so that each $I_i=I+(f_1,\ldots,f_i)$ is a radical ideal by Bertini's Theorem. Finally, since the dimension of $\cV(I_i)$ is exactly $k-i$, the $f_1,\ldots,f_k$ have to form a regular sequence modulo $I$ by \Cref{cor:regseq}.
\end{proof}

\begin{rem}
In the proof of the preceding lemma lies the reason why, in this section, we get lower bounds only for sufficiently large $d$. Namely, Scheiderer's result in \cite{scheider00}, which states that for every smooth algebraic curve $X$ there are polynomials of degree $d$ that have only real zeros on $X$, is only true for sufficiently large $d$. To find an explicit lower bound on $d$, that ensures the existence of such polynomials, is an open problem, except for the case of $M$-curves where a good lower bound has been provided by Huisman \cite{huisman01}.
\end{rem}

\begin{bsp}
The assumption on the singular locus in \Cref{lem:totreal} is necessary. Consider for example the singular plane curve $V_0=\cV(x^4-y^3)\subset\rset^2$. It is the image of the map $\rset\to\rset^2,t\mapsto(t^3,t^4)$. Now the zeros of a polynomial $f\in\rset[x,y]$ of degree $d$ on $V_0$ correspond to the roots of the univariate polynomial $f(t^3,t^4)$. But by Descartes' rule of signs this can not have $4d$ different real zeros. We dealt with curves of this kind in \Cref{sec:gaps}.
\end{bsp}

From this we get our main theorem on the Carath\'eodory numbers for measures supported on an algebraic set.

\begin{thm}\label{thm:maincara}
Let $\cX=V_0\subset\rset^n$ be Zariski closed of dimension $k>0$ such that its projective closure $V\subset\mathbb{P}^n$ is irreducible and its singular locus has codimension at least $2$. Let $P\in\qset[t]$ be the Hilbert polynomial of $V$. For large enough $d>0$, every moment functional $L:\rset[\cX]_{\leq 2d}\to\rset$ is a conic combination of at most $P(2d) -1$ point evaluations $l_{x_i}$ with $x_i\in\cX$. On the other hand, there are moment functionals $L:\rset[\cX]_{\leq 2d}\to\rset$ that are not a conic combination of fewer than
\[P(2d)-k\cdot P(d)+\binom k2\]
point evaluations $l_{x_i}$.
\end{thm}
\begin{proof}
For large enough $d$ we have $P(2d)=HF_V(2d)$. Therefore, by \Cref{prop:upper}, in order to prove the upper bound, it suffices to show that $P(2d)$ exceeds the number $m$ of path-connected components of $\cX$ for large enough $d$. Since $m$ is finite by \cite[Thm.\ 2.4.5, Prop.\ 2.5.13]{bochnak98}, this is clear because $P$ has positive degree $k$. For the lower bound we consider the polynomials $f_1,\ldots,f_k$ of degree $d$ from \Cref{lem:totreal} whose common zero set $Z$ on $\cX$ consists of $d^k\cdot\deg V$ simple points. The polynomial
\[f_1(1,x_1,\ldots,x_n)^2+\ldots+f_k(1,x_1,\ldots,x_n)^2\]
is nonnegative and has the same zero set $Z$ on $\cX$. By \Cref{thm:caraLowerZeroSet} a lower bound is then given by the dimension of the span of the point evaluations of polynomials of degree at most $2d$ in $Z$. This is the same as the dimension of the vector space $(\rset[x_0,\ldots,x_n]/J)_{2d}$ where $J$ is the homogeneous vanishing ideal of $Z$ considered as a subset of $\pset^n$. This is by definition $HF_Z(2d)$. Since $J$ is given by $I+(f_1,\ldots,f_k)$ where $I$ is the homogeneous vanishing ideal of $V$, and since the $f_i$ form a regular sequence, we have
\[HF_Z(2d)=HF_{I+(f_1,\ldots,f_k)}(2d)=\sum_{i=0}^k(-1)^i\binom ki HF_I(d\cdot(2-i))\]
by \Cref{lem:hilfun}. It follows immediately from the definition of the Hilbert function that $HF_I(m)=0$ for $m<0$ and $HF_I(0)=1$. Therefore, only the first three terms of the above sum are nonzero and we obtain:
\[HF_Z(2d)=HF_I(2d)-k HF_I(d)+\binom k2 .\]
For large $d$ the Hilbert function coincides with the Hilbert polynomial which shows the claim.
\end{proof}

\begin{exm}
This example is to demonstrate that both upper and lower bound from \Cref{thm:maincara} are false when $d$ is not large enough.
 Consider the plane curve $\cX\subset\rset^2$ defined as the zero set of the polynomial $x^8+y^8-1$. It is path-connected and its Zariski closure $C\subset\pset^2$ is smooth with Hilbert polynomial $P(d)=8d-20$. Thus for $d=1$ we have $$P(2d)-1=P(2)-1=-5<0$$which cannot be an upper bound. However, by \Cref{prop:upper} an upper bound in the case $d=1$ is given by $$HF_C(2d)-1=HF_C(2)-1=5.$$Finally, again for $d=1$, we have $$P(2d)-k\cdot P(d)+\binom k2=P(2)-P(1)=8 $$which exceeds the upper bound and thus cannot be a lower bound.
\end{exm}

\begin{exm}
Let $\cX\subset\rset^n$, $n\geq2$, be the boundary of the unit ball, i.e., the zero set of $1-(x_1^2+\ldots+x_n^2)$. Its Zariski closure $$V=\cV(x_0^2-(x_1^2+\ldots+x_n^2))\subset\pset^n$$ is irreducible and smooth. A direct computation gives its Hilbert polynomial: $$P(d)=\binom{n+d-1}{d}+\binom{n+d-2}{d-1}$$ It agrees with its Hilbert function for $d\geq0$. By \Cref{thm:maincara} there is a $d_0$ such that for all $d\geq d_0$, every moment functional $L:\rset[\cX]_{\leq 2d}\to\rset$ is a conic combination of at most $$\binom{n+2d-1}{2d}+\binom{n+2d-2}{2d-1}-1$$ point evaluations with points in $\cX$. Moreover, there are moment functionals $L:\rset[\cX]_{\leq 2d}\to\rset$ that are not a conic combination of fewer than
\[\binom{n+2d-1}{2d}+\binom{n+2d-2}{2d-1}-(n-1)\left(\binom{n+d-1}{d}+\binom{n+d-2}{d-1}\right)+\binom {n-1}2\]
point evaluations. We claim that in this case we can even choose $d_0=1$. Indeed, since $\cX$ is path-connected, $P(2d)$ exceeds the number of connected components whenever $d>0$. Moreover, letting $l_i=x_i$ the curve $X$ in the proof of \Cref{lem:totreal} is the Zariski closure of the unit circle in the plane. Then clearly for any $d\geq 1$ there is a polynomial of degree $d$ all of whose zeros on $X$ are real, simple and do not lie at the hyperplane at infinity: Consider for example the union of $d$ distinct lines through the origin. Thus the proof of \Cref{thm:maincara} shows that we can choose $d_0=1$. In the case $n=3$ the Carath\'eodory number $C$ is thus bounded by $$2d^2\leq C\leq4d(d+1).$$
\end{exm}

Let us examine the ratio of the lower and upper bound from \Cref{thm:maincara} as $d$ goes to infinity:
\begin{equation}\label{eq:hilbertpolynomialLimit}
\frac{P(2d)-k\cdot P(d)+\binom k2}{P(2d)} \quad=\quad 1-k\frac{P(d)}{P(2d)}+\frac{\binom k2}{P(2d)} \quad\xrightarrow{d\to\infty}\quad 1-\frac{k}{2^k}.
\end{equation}
Thus if the dimension $k$ of $\cX$ is not too small, our bounds are rather tight -- at least for large $d$. On the other hand, if $k=1$, i.e., $\cX$ is a smooth algebraic curve, using a refined argument, we obtain bounds that are even better, namely they differ only by one.

\begin{thm}\label{thm:caracurves}
Let $\cX=V_0\subset\rset^n$ be a compact algebraic set of dimension $1$ such that its projective closure $V\subset\mathbb{P}^n$ is a smooth irreducible curve of degree $e$. For large enough $d>0$, every moment functional $L:\rset[\cX]_{\leq 2d}\to\rset$ is a conic combination of at most $d\cdot e+1$ point evaluations $l_{x_i}$ with $x_i\in\cX$. On the other hand, there are moment functionals $L:\rset[\cX]_{\leq 2d}\to\rset$ that are not a conic combination of fewer than $d\cdot e$ point evaluations $l_{x_i}$.
\end{thm}
\begin{proof}
The Hilbert polynomial of $V$ is of the form $HP_V(t)=e\cdot t+a$. Thus the lower bound from \Cref{thm:maincara} is just $d\cdot e$. 

In order to prove the upper bound we use a similar technique as in \Cref{lem:upperGabs}. At first we show that a nonnegative polynomial $f$ on $\cX$ of degree $2d$ can have at most $d\cdot e$ different zeros on $\cX$ (or vanishes on all of $\cX$). Indeed, the zero set of $f$ on $\cX$ is contained in $\cV(f^h)\cap V$ where $f^h$ is the homogenization of $f$. Since $\cV(f^h)$ is a hypersurface of degree $2d$ and $V$ a curve of degree $e$, the intersection $\cV(f^h)\cap V$ consists of $2d e$ points counted with multiplicity. But because $f$ is nonnegative on $\cX$, each zero of $f$ on $\cX$ must have even multiplicity as otherwise there would be a sign change. This shows that $f$ has at most $d\cdot e$ different zeros on $\cX$.

Now we show the upper bound $d\cdot e + 1$. Let $\sA$ be a basis of $\rset[\cX]_{\leq 2d}$ and $s\in\cS_\sA$ be the moment sequence of $L$. Since $1\in\rset[\cX]_{\leq 2d}$ and $\cX$ is compact, $\cS_\sA$ is closed and pointed, i.e.,
\[c_s(x) := \sup \{c\geq 0 \,|\, s - c\cdot s_\sA(x)\in\cS_\sA\} < \infty\]
is attained for every $x\in\cX$. Hence, $s' = s - c_s(x)\cdot s_\sA(x)\in\partial\cS_\sA$ and there exists an $f\in\rset[\cX]_{\leq 2d}$ such that $f\geq 0$ on $\cX$ and $L_{s'}(f) = 0$ holds. Since $f$ has at most $d\cdot e$ zeros, $s'$ is represented by at most $d\cdot e$ point evaluations. $s = s' + c_s(x)\cdot s_\sA(x)$ requires therefore at most $d\cdot e + 1$ point evaluations in $\cX$.
\end{proof}

\begin{rem}
As we have seen in \Cref{sec:gaps}, the smoothness assumption in \Cref{thm:caracurves} is crucial.
\end{rem}


In the next section we use the techniques and results from this and the preceding sections to obtain new lower bounds for the cases $\cX = \rset^n$ and $\cX = [0,1]^n$.

\section{Lower Bounds on the Carath\'eodory Number}
\label{sec:cara}

Several lower bounds on the Carath\'eodory number are known, see e.g.\ \cite{davis84}. For bivariate polynomials of odd degree $\cA = \rset[x_1,x_2]_{\leq 2d-1}$ M\"oller \cite{moller76} proved
\[\begin{pmatrix} d+1\\ 2\end{pmatrix} + \left\lfloor\frac{d}{2}\right\rfloor \quad\leq\quad \cat_{\sA_{2,2d-1}}.\]
In \cite{didio17Cara} we gave a very general lower bound improving M\"ollers lower bound to
\[\left\lceil\frac{1}{3}\begin{pmatrix} 2d+1\\ 2\end{pmatrix}\right\rceil \quad\leq\quad \cat_{\sA_{2,2d-1}} \qquad\text{and}\qquad \left\lceil\frac{1}{3}\begin{pmatrix} 2d+2\\ 2\end{pmatrix}\right\rceil \quad\leq\quad \cat_{\sA_{2,2d}}.\]
In \cite{rienerOptima} C.\ Riener and M.\ Schweighofer further improved the lower bound to
\begin{equation}\label{eq:rienerLowerBound}
(d-1)^2 \quad\leq\quad \cat_{\sA_{2,2d-1}}.
\end{equation}
They used \cite[Prop.\ 8.5]{rienerOptima}, a polynomial version of \Cref{thm:caraLowerZeroSet}, applied to $f_1^2+f_2^2$ where
\[f_1(x) = (x-1)(x-2)\cdots (x-d) \quad\text{and}\quad f_2(y) = (y-1)(y-2)\cdots (y-d)\]
and found $\dim \rset[x,y]/(f_1,f_2) = d^2$, i.e.,
$\dim\lin \{s_\sA(x_i,y_j) \,|\, x_i,y_j= 1,\dots,d\} = d^2$
and therefore the moment functional $L:\rset[x,y]_{\leq 2d}\rightarrow\rset$ with $L = \sum_{i,j=1}^d l_{(i,j)}$ has Carath\'eodory number $d^2$. In \cite{didioConeArXiv} this was extended to higher dimensions by investigating the linear (in)dependence of $s_\sA(x_i)$ on the grid $G = \{1,\dots,d\}^n$ (for $\cX = \rset^n$) and $G = \{0,1,\dots,d\}^n$ (for $\cX = [0,d]^n$). As in the previous section the main idea is that the dimension of point evaluations
\begin{equation}\label{eq:pointEvalDim}
\dim\lin\{s_{\sA_{n,d}}(x) \,|\, x\in \cZ(f)\} 
\end{equation}
can be translated into
\begin{equation}\label{eq:moduloDim}
\dim (\rset[x_0,\dots,x_n]/I)_d,
\end{equation}
i.e., the dimension of the homogeneous part of $\rset[x_0,\dots,x_n]/I$ of degree $d$ for some homogeneous ideal $I$.

\begin{lem}\label{lem:pqResults}
Let $n,d\in\nset$ and set
\[p_i = (x_i - x_0)\cdots (x_i - dx_0) \quad\text{and}\quad q_i = x_i(x_i - x_0)\cdots (x_i - dx_0)\]
for $i=1,\dots,n$. The following holds:
\begin{enumerate}[i)]
\item The sequences $p_1,\dots,p_n$ and $q_1,\dots,q_n$ are regular.

\item The ideals generated by $p_1,\dots,p_n$ resp. $q_1,\dots,q_n$ are radical.

\item Let $f_1,\dots,f_n$ be a regular sequence of homogeneous functions $f_i$ of degree $d$. The Hilbert function $HF_{R_n}$ of $R_n := \rset[x_0,\dots,x_n]/(f_1,\dots,f_n)$ is
\[HF_{R_n}(k)=\sum_{i=0}^n (-1)^i\cdot\binom ni \cdot HF_{\pset^n}(k-i\cdot d).\]
In particular, we have
\begin{align*}
HF_{R_n}(2d-2) &= \begin{pmatrix} n+2d-2\\ n\end{pmatrix} - n\cdot \begin{pmatrix} n+d-2\\ n\end{pmatrix},\\
HF_{R_n}(2d-1) &= \begin{pmatrix} n+2d-1\\ n\end{pmatrix} - n\cdot \begin{pmatrix} n+d-1\\ n\end{pmatrix},\\
HF_{R_n}(2d) &= \begin{pmatrix} n+2d\\ n\end{pmatrix} - n\cdot \begin{pmatrix} n+d\\ n\end{pmatrix} + \begin{pmatrix} n\\ 2\end{pmatrix},
\intertext{and}
HF_{R_n}(2d+1) &= \begin{pmatrix} n+2d+1\\ n\end{pmatrix} - n\cdot \begin{pmatrix} n+d+1\\ n\end{pmatrix} + 3\cdot\begin{pmatrix} n+1\\ 3\end{pmatrix}.
\end{align*}
\end{enumerate}
\end{lem}
\begin{proof} Part (i) follows directly from the fact that each $p_i$ resp.\ $q_i$ is a monic polynomials over $\rset[x_0]$ in the single variable $x_i$. Part (ii) is a direct consequence of \cite[Thm. 1.1]{alon99}. Finally, since $HF_{\pset^n}(k)=\binom{n+k}{k}$ for $k\geq0$ and $HF_{\pset^n}(k)=0$ otherwise, \Cref{lem:hilfun} directly implies (iii).
\end{proof}

From this lemma we derive the following lower bounds for the Carath\'eodory number $\cat_{\sA_{n,2d}}$ and $\cat_{\sA_{n,2d+1}}$ on $\cX = \rset^n$.

\begin{thm}\label{thm:lowerboundsRn}
Let $n,d\in\nset$ and $\cX\subseteq\rset^n$ with non-empty interior. For even degree $\cA = \rset[x_1,\dots,x_n]_{\leq 2d}$  we have
\[\cat_{\sA_{n,2d}} \quad\geq\quad \begin{pmatrix} n+2d\\ n\end{pmatrix} - n\cdot \begin{pmatrix} n+d\\ n\end{pmatrix} + \begin{pmatrix} n\\ 2\end{pmatrix}\]
and for odd degree $\cA = \rset[x_1,\dots,x_n]_{\leq 2d+1}$ we have
\[\cat_{\sA_{n,2d+1}} \quad\geq\quad \begin{pmatrix} n+2d+1\\ n\end{pmatrix} - n\cdot \begin{pmatrix} n+d+1\\ n\end{pmatrix} + 3\cdot \begin{pmatrix} n+1\\ 3\end{pmatrix}.\]
\end{thm}
\begin{proof}
Since $\cX\subseteq\rset^n$ has non-empty interior there is a $\varepsilon > 0$ and $y\in\rset^n$ such that $y + \varepsilon\cdot \{1,\dots,d\}^n\subset\cX$. The affine map
$T:\cX'\rightarrow\cX,\ x\mapsto y+\varepsilon\cdot x$
shifts the moment problem on $\cX$ to $\cX' = \varepsilon^{-1}\cdot (\cX-y)$ with $\rset[x_1,\dots,x_n]_{\leq dD} = \rset[x_1,\dots,x_n]_{\leq D}\circ T$, $D = 2d, 2d+1$, and $\{1,\dots,d\}^n\subset\cX'$. So w.l.o.g.\ we can assume that $\{1,\dots,d\}^n\subset\cX$. Then we can proceed as in the proof of \Cref{thm:maincara} by choosing the $f_i$ to be the $p_i$ from \Cref{lem:pqResults}. We have already calculated the concrete resulting values of the Hilbert function in \Cref{lem:pqResults}.
\end{proof}

These lower bounds coincide with the numerical results in \cite[Tab.\ 2]{didioConeArXiv}. Note that for $n=1$ we get for the even and odd degree cases the bound $d$. This is the maximal number of zeros of a non-zero and non-negative univariate polynomial, i.e., the Carath\'eodory number of moment sequences on the boundary of the moment cone $\cS_{\sA_{n,2d}}$ or $\cS_{\sA_{n,2d+1}}$, respectively. In fact, we proved the following.

\begin{prop}\label{prop:momentOnBoundary}
Let $n,d\in\nset$, $k\in\{0,1\}$, $\cX=\rset^n$, and $G=\{1,\dots,d\}^n$. Then
\[s = \sum_{x\in G} s_{\sA_{n,2d+k}}(x) \qquad\text{resp.}\qquad L = \sum_{x\in G} l_x:\rset[x_1,\dots,x_n]_{\leq 2d+k}\rightarrow\rset\]
supported on the grid $G$ with $L(p)=0$, $p=p_1^2+ \cdots + p_n^2\geq 0$ from \Cref{lem:pqResults}, and the representing measure $\mu = \sum_{x\in G} \delta_x$ has the Carath\'eodory number
\[\cat_{\sA_{n,2d+k}}(s) = \begin{cases}
\left(\begin{smallmatrix} n+2d\\ n\end{smallmatrix}\right) - n\cdot \left(\begin{smallmatrix} n+d\\ n\end{smallmatrix}\right) + \left(\begin{smallmatrix} n\\ 2\end{smallmatrix}\right) & \text{for}\ k=0,\phantom{\Big(}\\
\left(\begin{smallmatrix} n+2d+1\\ n\end{smallmatrix}\right) - n\cdot \left(\begin{smallmatrix} n+d+1\\ n\end{smallmatrix}\right) + 3\cdot \left(\begin{smallmatrix} n+1\\ 3\end{smallmatrix}\right) & \text{for}\ k=1.\phantom{\Big(} \end{cases}\]
\end{prop}

We get the following lower bounds for the case $\cX = [0,1]^n$ (or equivalently $\cX = [0,d]^n$) which serves as an example of a compact set $\cX$.

\begin{thm}\label{thm:lowerboundsCube}
Let $n,d\in\nset$ and $\cX = [0,1]^n$. For even degree $\cA = \rset[x_1,\dots,x_n]_{\leq 2d}$ we have
\[\cat_{\sA_{n,2d}} \quad\geq\quad \begin{pmatrix} n+2d\\ n\end{pmatrix} - n\cdot \begin{pmatrix} n+d-1\\ n\end{pmatrix}\]
and for odd degree $\cA = \rset[x_1,\dots,x_n]_{\leq 2d+1}$ we have
\[\cat_{\sA_{n,2d+1}} \quad\geq\quad \begin{pmatrix} n+2d+1\\ n\end{pmatrix} - n\cdot \begin{pmatrix} n+d\\ n\end{pmatrix}.\]
\end{thm}
\begin{proof}
The proof follows the same arguments as in \Cref{thm:lowerboundsRn}. Since we work on $[0,1]^n$ we can choose the $f_i$'s to be the $q_i$'s as in \Cref{thm:maincara}. Then \Cref{lem:pqResults} provides the explicit values of the Carath\'eodory numbers.
\end{proof}

Additionally, note the difference between the lower bounds from \Cref{thm:lowerboundsRn} and \Cref{thm:lowerboundsCube}. In the one-dimensional case a non-negative polynomial $p$ of degree $2d$ has at most $d$ zeros by the fundamental theorem of algebra:
\[p(x) = (x-x_1)^2\cdots (x-x_d)^2.\]
However, on the interval $[0,1]$ a non-negative polynomial $q$ of degree $2d$ can have up to $d+1$ zeros
\[q(x) = (x-x_1)\cdot (x-x_2)^2\cdots (x-x_{d-1})^2\cdot (x_d-x)\]
when $x_1 = 0$ and $x_d = 1$ holds. So interior zeros count twice, zeros on the boundary only once. This concept appeared already in the classical works of Kre\u{\i}n and Nudel'man about $T$-systems, see \cite[Ch.\ 2]{kreinMarkovMomentProblem}. So in higher dimensions for $\rset^n$ all zeros are interior points but for $[0,1]^n$ we can place zeros on the boundary.

Note that for $n=1$ we get for the even and the odd case the lower bound $d+1$. This is the maximal number of zeros of a non-zero and non-negative polynomial on $[0,1]$. For $n=2$ we get the following.

\begin{cor}
For $d\in\nset$ and $\cX = [0,1]^2$ ($n=2$) we have
\[\cat_{\sA_{2,2d}} \;\geq\; (d+1)^2 \qquad\text{and}\qquad \cat_{\sA_{2,2d+1}} \;\geq\; (d+1)^2.\]
\end{cor}

Theorems \ref{thm:lowerboundsRn} and \ref{thm:lowerboundsCube} give lower bounds on the Carath\'eodory number of $\cS_{\sA_{n,k}}$ by constructing one specific boundary moment sequence $s$ and calculating its Cara\-th\'eo\-dory number $\cat_{\sA_{n,k}}(s)$. But from the following considerations it will be clear that Theorems \ref{thm:lowerboundsRn} and \ref{thm:lowerboundsCube} already show that in higher dimensions and degrees the Carath\'eodory numbers behave very badly, see \Cref{thm:lowerboundLimits}. Previous results in \cite{didio17Cara} and \cite{rienerOptima} show that for $n=2$ we have
\begin{equation}\label{eq:limitsQuotTwoDim}
\frac{1}{2} \quad\leq\quad \liminf_{d\rightarrow\infty} \frac{\cat_{\sA_{2,d}}}{|\sA_{2,d}|}\quad\leq\quad \limsup_{d\rightarrow\infty} \frac{\cat_{\sA_{2,d}}}{|\sA_{2,d}|} \quad\leq\quad \frac{3}{4}.
\end{equation}
From Theorems \ref{thm:lowerboundsRn} and \ref{thm:lowerboundsCube} we get the following limits.

\begin{thm}\label{thm:lowerboundLimits}
For $\cX\subseteq\rset^n$ with non-empty interior we have
\begin{align*}
\liminf_{d\rightarrow\infty} \frac{\cat_{\sA_{n,d}}}{|\sA_{n,d}|}\quad &\geq\quad 1-\frac{n}{2^n} & \textrm{ for all } n &\in\nset\\
\intertext{and}
\lim_{n\rightarrow\infty} \frac{\cat_{\sA_{n,d}}}{|\sA_{n,d}|}\quad &= \quad 1 &\textrm{ for all } d&\in\nset.
\end{align*}
\end{thm}
\begin{proof}
Follows by a direct calculation as in \Cref{eq:hilbertpolynomialLimit}.
\end{proof}

In (\ref{eq:limitsQuotTwoDim}), i.e., \cite{didio17Cara} and \cite{rienerOptima}, we have seen that for $n=2$ the upper bound on the Carath\'eodory number is considerably smaller than $|\sA_{2,d}|$, namely $\frac{3}{4}\cdot |\sA_{2,d}|$ is an upper bound. But Theorems \ref{thm:lowerboundsRn}, \ref{thm:lowerboundsCube}, and \ref{thm:lowerboundLimits} confirm the apprehensions in \cite{didioConeArXiv} on the Carath\'eodory numbers and their limits. Note that for $\cX = [0,1]^n$ the following was proved already in \cite{didioConeArXiv}.

\begin{thm}[{\cite[Thm.\ 59]{didioConeArXiv}}]
For $\rset[x_1,\dots,x_n]_{\leq 2}$ on $\cX = [0,1]^n$ we have
\[\begin{pmatrix} n+2\\ 2\end{pmatrix} - n \quad\leq\quad \cat_{\sA_{n,2}} \quad\leq\quad \begin{pmatrix}n+2\\ 2\end{pmatrix} - 1.\]
\end{thm}

Thus for higher dimensions $n$, even with fixed degree $d$, it is not possible to give upper bounds $\cat_{\sA_{n,d}}\leq c\cdot |\sA_{n,d}|$ with $c<1$ for all $n$.

\begin{cor}
Let $\cX\subseteq\rset^n$ with non-empty interior and $d\in\nset$. For $\varepsilon>0$ there is an $N\in\nset$ such that for every $n\geq N$ there is a moment sequence $s\in\cS_{\sA_{n,d}}$ resp.\ a moment functional $L:\rset[x_1,\dots,x_n]_{\leq d}\rightarrow\rset$ with
\[\cat_{\sA_{n,d}}(s) \quad\geq\quad (1-\varepsilon)\cdot\begin{pmatrix} n+d\\ n\end{pmatrix},\]
i.e., $L_s$ is the conic combination of at least $(1-\varepsilon)\cdot\big(\begin{smallmatrix} n+d\\ n\end{smallmatrix}\big)$ point evaluations $l_{x_i}$.
\end{cor}
\begin{proof}
Choose $s$ resp.\ $L$ as in \Cref{prop:momentOnBoundary}. This has the desired property.
\end{proof}

So even when we work with the probably most well behaved moment problem, i.e., polynomials, the Carath\'eodory number is cursed by high dimensionalities. In the next section we study the consequences of these new lower bounds and their limits for Hankel matrices and flat extension.

\section{Hankel Matrices and Flat Extension}
\label{sec:flat}

Recall that for a finite dimensional space $\cA$ of measurable functions with basis $\sA = \{a_1,\dots,a_m\}$ the \emph{Hankel matrix} $\cH(L)$ of a linear functional $L:\cA^2\rightarrow\rset$ is given by $\cH(L) = (L(a_i a_j))_{i,j=1}^m$, i.e.,
\begin{equation}\label{eq:hankelIntegral}\small 
\cH(L) = \int_\cX \begin{pmatrix}
a_1(x) a_1(x) & \hdots & a_1(x) a_m(x)\\
\vdots & \ddots & \vdots\\
a_1(x) a_m(x) & \hdots & a_m(x) a_m(x)
\end{pmatrix}~\diff\mu(x)
= \int_\cX s_\sA(x)\cdot s_\sA(x)^T~\diff\mu(x)
\end{equation}
if $\mu$ is a (signed) representing measure of $L$. Hence we have the following.

\begin{lem}\label{lem:hankelRank}
Let $\cA$ be a finite dimensional vector space of measurable functions with basis $\sA =\{a_1,\dots,a_m\}$. For $L:\cA^2\rightarrow\rset$ with $L = \sum_{i=1}^k c_i\cdot l_{x_i}$ ($c_i\in\rset$) we have
\[\cH(L)=\sum_{i=1}^k c_i\cdot s_\sA(x_i)\cdot s_\sA(x_i)^T \qquad\text{and}\qquad \rank\cH(L)\leq k.\]
The following are equivalent:
\begin{enumerate}[i)]
\item $\rank\cH(L) = k$.
\item $s_\sA(x_1),\dots,s_\sA(x_k)$ are linearly independent.
\end{enumerate}
\end{lem}
\begin{proof}
By replacing $x^\alpha$, $x^\beta$, and $x^{\alpha+\beta}$ by $a_i$, $a_j$, and $a_{i,j}=a_i\cdot a_j$, respectively, with (\ref{eq:hankelIntegral}) the proof is verbatim the same as in \cite[Prop.\ 17.21]{schmudMomentBook}.
\end{proof}

Note that the previous result holds for signed representing measures.

It is very hard to check whether a linear functional $L$ is a moment functional. If $\cX$ is compact and $\cA^2$ contains an $e>0$ on $\cX$ then one has to check the following condition:
\[L_s(p)\geq 0 \quad\forall p\in\pos(\cA^2,\cX) := \{a\in\cA^2 \,|\, a\geq 0\}.\]
But the set of non-negative functions $\pos(\cA^2,\cX)$ on $\cX$ is in general hard to describe. For example deciding, whether a polynomial $p\in\rset[x_1,\ldots,x_n]_{\leq2d}$ is non-negative, is an NP-hard problem (for fixed $d\geq2$ as a function of $n$), see e.g.\ \cite[p.56]{blekSemiOpt}. One approach to overcome this problem is to approximate non-negative polynomials with sums of squares (SOS): Checking whether a given polynomial is a sum of squares is equivalent to deciding whether a certain semidefinite program (SDP) is feasible, see \cite[\S 4.1.4]{blekSemiOpt}, and (under some mild assumptions) one can solve an SDP up to a fixed precision in time that is polynomial in the program description size \cite[\S 2.3.1]{blekSemiOpt}. The connection between the truncated moment problem and non-negative polynomials run deep and can be found in a large number of publications on the truncated moment problem, see e.g.\ \cite{akhiezClassical,kreinMarkovMomentProblem,marshallPosPoly,lauren09,lasserreSemiAlgOpt,schmudMomentBook}.

\emph{Flat extension} is another method to determine whether a linear functional $L:\rset[x_1,\dots,x_n]_{\leq 2d}\rightarrow\rset$ is a moment functional, see e.g.\ \cite{curto2,curto3,lauren09,lauren09a,schmudMomentBook}. Let $D\geq d$ and $L_0:\rset[x_1,\dots,x_n]_{\leq 2D}\rightarrow\rset$ be a linear functional that extends $L$. An extension $L_1:\rset[x_1,\dots,x_n]_{\leq 2D+2}$ of $L_0$ is called \emph{flat with respect to $L_0$} if $\rank \cH(L_1) = \rank \cH(L_0)$. Then by the flat extension theorem, see \cite{curto2}, \cite{curto3} or e.g.\ \cite[Thm.\ 17.35]{schmudMomentBook}, there are linear functionals $L_i:\rset[x_1,\dots,x_n]_{\leq 2D+2i}$ with $\rank(L_0)=\rank(L_i)$ such that $L_{i}$ extends $L_{i-1}$ for all $i\in\nset$. These determine a linear functional $L_\infty:\rset[x_1,\dots,x_n]\rightarrow\rset$ which is called a \emph{flat extension of $L_0$} (\emph{to all $\rset[x_1,\dots,x_n]$}), i.e., every restriction $L_\infty|_{\rset[x_1,\dots,x_n]_{\leq 2D'+2}}$ is a flat extension of $L_\infty|_{\rset[x_1,\dots,x_n]_{\leq 2D'}}$ for all $D'\geq D$. Exists such a flat extension $L_\infty$, then by the flat extension theorem $L_0$ is a moment functional if $L_0(a^2)\geq 0$ for all $a\in\rset[x_1,\dots,x_n]_{\leq D}$. In this case $L$ is of course a moment functional as well. It was open to which degree $2D$ the functional $L$ must be extended in order to have a flat extension. The upper bound of $D\leq 2d$ follows immediately from the Carath\'eodory bound \cite{curto2,curto3}. Part one of the following theorem is due to Curto and Fialkow. Our new lower bounds on the Carath\'eodory number show that $D = 2d$ is attained and is stated in part two of the following theorem.

\begin{thm}\label{thm:flatExtensionBound}\
\begin{enumerate}[i)]
\item For every moment functional $L:\rset[x_1,\dots,x_n]_{\leq 2d}\rightarrow\rset$ there is a $D\leq 2d$ and an extension to a moment functional \mbox{$L_0:\rset[x_1,\dots,x_n]_{\leq 2D}\rightarrow\rset$} that admits a flat extension $L_{\infty}:\rset[x_1,\dots,x_n]\to\rset$.

\item For every $d\in\nset$ there is an $N\in\nset$ such that for every $n\geq N$ there is a moment functional $L$ on $\rset[x_1,\dots,x_n]_{\leq 2d}$ {such that $D=2d$ in (i) is required}.
\end{enumerate}
\end{thm}
\begin{proof}
(i): By \cite[Cor.\ 14]{didio17Cara} we have $C:=\cat_{\sA_{n,2d}}(L)\leq \left(\begin{smallmatrix} n+2d\\ n\end{smallmatrix}\right)-1$, i.e.,
\[L = \sum_{i=1}^{C} c_i\cdot {l}_{x_i} \quad\text{with}\quad c_i>0\]
and the $l_{x_i}$ are linearly independent on $\rset[x_1,\dots,x_n]_{\leq 2d}$. Then $L_{\infty}:\rset[x_1,\dots,x_n]$ $\to\rset$ defined by $L_{\infty}(f):=\sum_{i=1}^{C} c_i\cdot {f}({x_i})$ is a flat extension of $L_0:=L|_{\rset[x_1,\dots,x_n]_{\leq 4d}}$.

(ii): Let $s\in\cS_{\sA_{n,2d}}$ resp.\ $L_s:\rset[x_1,\dots,x_n]_{\leq 2d}\rightarrow\rset$ as in \Cref{prop:momentOnBoundary} and assume $D = 2d - c$, $c\in\nset$. From the condition $\cat_{\sA_{n,2d}}(s) \leq \left(\begin{smallmatrix} n+D\\ n\end{smallmatrix}\right)$ that the Hankel matrix of the flat extension must be at least the size of the Carath\'eodory number of $s$ we find that
\begin{multline*}
1 \leq \lim_{n\rightarrow\infty} \frac{\begin{pmatrix}n + 2d - c\\ n\end{pmatrix}}{\cat_{\sA_{n,2d}}(s)}
= \lim_{n\rightarrow\infty} \frac{\begin{pmatrix}n + 2d - c\\ n\end{pmatrix}}{\begin{pmatrix} n + 2d\\ n\end{pmatrix}} \cdot \underbrace{\frac{\begin{pmatrix} n+2d\\ n\end{pmatrix}}{\cat_{\sA_{n,2d}}(s)}}_{\rightarrow 1\ \text{by Thm.\ \ref{thm:lowerboundLimits}}}\\
= \lim_{n\rightarrow\infty} \frac{(2d-c+1)\cdots (2d)}{(n+2d-c+1)\cdots (n+2d)} = 0.
\end{multline*}
This is a contradiction, i.e., $c=0$ must hold.
\end{proof}

\begin{exm}
Consider the moment sequences $s = (s_\alpha)_{\alpha\in\nset_0^n:|\alpha|\leq 2d}$ resp.\ functionals $L:\rset[x_1,\dots,x_n]_{\leq 2d}\to\rset$ from \Cref{prop:momentOnBoundary} supported on the grid $\{1,\dots,d\}^n$. The condition $\cat_{\sA_{n,2d}}(s) \leq \left(\begin{smallmatrix} n+D\\ n\end{smallmatrix}\right)$, meaning that the size of the Hankel matrix must be at least the size of Carath\'eodory number of $s$, shows that $(n,d)$ = $(9,2)$, $(7,3)$, $(6,4)$, $(6,5)$, and $(n',6)$, for all $n'\geq 6$, are small examples where the worst case extension to degree $D=2d$ is attained. Even for $d = 10^{15}$ the worst case extension is already necessary for $n=51$.
\end{exm}


\begin{thebibliography}{MVKW95}

\bibitem[Akh65]{akhiezClassical}
N.~I. Akhiezer, \emph{The classical moment problem and some related questions
  in analysis}, Oliver \& Boyd, Edinburgh, 1965.

\bibitem[Alo99]{alon99}
N.~Alon, \emph{Combinatorial {N}ullstellensatz}, Combin.\ Probab.\ Comput.
  \textbf{8} (1999), 7--29.

\bibitem[AM69]{atiyahIntroComAlg}
M.~F. Atiyah and I.~G. Macdonald, \emph{Introduction to commutative algebra},
  Addison-Wesley Publishing Co., Reading, Mass.-London-Don Mills, Ont., 1969.

\bibitem[BCR98]{bochnak98}
J.~Bochnak, M.\ Coste, and M.-F.\ Roy, \emph{Real {A}lgebraic {G}eometry},
  Springer-Verlag, Berlin, 1998.

\bibitem[BH93]{BrunsCohenMacaulay}
W.~Bruns and J.~Herzog, \emph{Cohen-{M}acaulay rings}, Cambridge University
  Press, Cambridge, 1993.

\bibitem[BJL19]{brehard19}
F.~Br\'ehard, M.~Joldes, and J.-B. Lasserre, \emph{On a moment problem with
  holonomic functions}, {ISSAC} '19: {P}roceedings of the 2019 on
  {I}nternational {S}ymposium on {S}ymbolic and {A}lgebraic {C}omputation,
  2019, pp.~66--73.

\bibitem[BPT13]{blekSemiOpt}
G.~Blekherman, P.~A. Parrilo, and R.~R. Thomas (eds.), \emph{Semidefinite
  optimization and convex algebraic geometry}, MOS-SIAM Series on Optimization,
  vol.~13, Society for Industrial and Applied Mathematics (SIAM), Philadelphia,
  PA; Mathematical Optimization Society, Philadelphia, PA, 2013.

\bibitem[CF96]{curto2}
R.~Curto and L.~A. Fialkow, \emph{Solution of the truncated moment problem for
  flat data}, Mem. Amer. Math. Soc. \textbf{119} (1996), no.~568.

\bibitem[CF98]{curto3}
\bysame, \emph{Flat extensions of positive moment matrices: recursively
  generated relations}, Mem. Amer. Math. Soc. \textbf{136} (1998), no.~648.

\bibitem[CF05]{curto05}
\bysame, \emph{Truncated {$K$}-moment problems in several variables}, J.~Op.\
  Theory \textbf{54} (2005), 189--226.

\bibitem[DBN92]{dai92}
M.~Dai, P.~Baylou, and M.~Najim, \emph{An efficient algorithm for computation
  of shape moments from run-length codes or chain codes}, Pattern Recognit.
  \textbf{25} (1992), 1119--1128.

\bibitem[dD18]{didioDiss}
P.~J. di~Dio, \emph{The truncated moment problem}, Ph.D. thesis, University of
  Leipzig, 15 Jan 2018.

\bibitem[dD19a]{didio18gaussian}
\bysame, \emph{The multidimensional truncated {M}oment {P}roblem: Gaussian and
  {L}og-{N}ormal {M}ixtures, their {C}arath\'eodory {N}umbers, and {S}et of
  {A}toms}, Proc.\ Amer.\ Math.\ Soc. \textbf{147} (2019), 3021--3038.

\bibitem[dD19b]{didio19DerivMom}
\bysame, \emph{The multidimensional truncated {M}oment {P}roblem: {S}hape and
  {G}aussian {M}ixture {R}econstruction from {D}erivatives of {M}oments},
  arXiv:1903.00790.

\bibitem[dDS18a]{didioConeArXiv}
P.~J. di~Dio and K.~Schm\"{u}dgen, \emph{The multidimensional truncated moment
  problem: The moment cone} (2018), submitted, https://arxiv.org/abs/1809.00584.

\bibitem[dDS18b]{didio17Cara}
P.~J. di~Dio and K.~Schm{\"u}dgen, \emph{{The} multidimensional truncated
  {M}oment {P}roblem: {C}arath\'eodory {N}umbers}, J.~Math.\ Anal.\ Appl.
  \textbf{461} (2018), 1606--1638.

\bibitem[DR84]{davis84}
P.~J. Davis and P.~Rabinowitz, \emph{Methods of {N}umerical {I}ntegration}, 2nd
  ed., Academic Press, Orlando, FL, 1984.

\bibitem[Eis95]{eisenComAlg}
D.~Eisenbud, \emph{Commutative algebra: With a view toward algebraic geometry},
  Springer-Verlag, New York, 1995.

\bibitem[Eis05]{eisenGeomSyz}
\bysame, \emph{The geometry of syzygies: A second course in commutative algebra
  and algebraic geometry}, Springer-Verlag, New York, 2005.

\bibitem[Fia16]{fialkoMomProbSurv}
L.~A. Fialkow, \emph{The truncated {$K$}-moment problem: a survey}, Theta Ser.\
  Adv.\ Math. \textbf{18} (2016), 25--51.

\bibitem[Fia17]{fialkow17}
\bysame, \emph{The core variety of a multi-sequence in the truncated moment
  problem}, J.~Math.\ Anal.\ Appl. \textbf{456} (2017), 946--969.

\bibitem[FN10]{fialkow10}
L.~A. Fialkow and J.~Nie., \emph{Positivity of {R}iesz functionals and
  solutions of quadratic and quartic moment problems}, J.~Funct.\ Anal.
  \textbf{258} (2010), 328--356.

\bibitem[GLPR12]{gravin12}
N.~Gravin, J.~Lasserre, D.~V. Pasechnik, and S.~Robins, \emph{The inverse
  moment problem for convex polytopes}, Discrete Comput.\ Geom. \textbf{48}
  (2012), 596--621.

\bibitem[GMV99]{golub99}
G.~H. Golub, P.~Milfar, and J.~Varah, \emph{A stable numberical method for
  inverting shape from moments}, SIAM J.\ Sci.\ Comput. \textbf{21} (1999),
  no.~4, 1222--1243.

\bibitem[Gra99]{grab99}
D.~J. Grabiner, \emph{Descartes' rule of signs: another construction}, Amer.\
  Math.\ Monthly \textbf{106} (1999), 854--856.

\bibitem[HN12]{helton12}
J.~W. Helton and J.~Nie, \emph{A Semidefinite Approach for Truncated $K$-Moment
  Problems}, Found.\ Comput.\ Math.\ \textbf{12} (2012), 851--881.

\bibitem[Hui01]{huisman01}
J.~Huisman, \emph{On the geometry of algebraic curves having many real
  components}, Rev.\ Mat.\ Complut. \textbf{14} (2001), 83--92.

\bibitem[Jou83]{jouano83}
J.-P. Jouanolou, \emph{Th\'eor\'emes de {B}ertini et applications},
  Birkh\"auser Boston, Inc., Boston, MA, 1983.

\bibitem[Kem68]{kemper68}
J.~H.~B. Kemperman, \emph{The {G}eneral {M}oment {P}roblem, a {G}eometric
  {A}pproach}, Ann.\ Math.\ Stat. \textbf{39} (1968), 93--122.

\bibitem[Kem87]{kemper87}
\bysame, \emph{Geometriy of the moment problem}, Proc.\ Sym.\ Appl.\ Math.
  \textbf{37} (1987), 16--53.

\bibitem[KN77]{kreinMarkovMomentProblem}
M.~G. Kre\u{\i}n and A.~A. Nudel'man, \emph{The {M}arkow {M}oment {P}roblem and
  {E}xtremal {P}roblems}, American Mathematical Society, Providence, Rhode
  Island, 1977.

\bibitem[Kun14]{kunertPhD14}
A.~Kunert, \emph{Facial {S}tructure of {C}ones of non-negative {F}orms}, Ph.D.
  thesis, Universtity of Konstanz, 2014.

\bibitem[Lan80]{landauMomAMSProc}
H.~J. Landau (ed.), \emph{Moments in {M}athematics}, Proceedings of Symposia in
  applied Mathematics, vol.~37, Providence, RI, American Mathematical Society,
  1980.

\bibitem[Las15]{lasserreSemiAlgOpt}
J.-B. Lasserre, \emph{An introduction to polynomial and semi-algebraic
  optimization}, Cambridge University Press, Cambridge, 2015.

\bibitem[Lau05]{laurent05}
M.~Laurent, \emph{Revisiting two theorems of {C}urto and {F}ialkow on moment
  matrices}, Proc.\ Amer.\ Math.\ Soc. \textbf{133} (2005), 2965--2975.

\bibitem[Lau09]{lauren09}
\bysame, \emph{Sums of {S}quares, {M}oment {M}atrices and {P}olynomial over
  {O}ptimization}, Emerging application of algebraic geometry, IMA Vol. Math.
  Appl., vol. 149, Springer, New York, 2009, pp.~157--270.

\bibitem[LM09]{lauren09a}
M.~Laurent and B.~Mourrain, \emph{A generalized flat extension theorem for
  moment matrices}, Arch.\ Math.\ (Basel) \textbf{93} (2009), no.~1, 87--98.

\bibitem[Mar08]{marshallPosPoly}
M.~Marshall, \emph{Positive {P}olynomials and {S}ums of {S}quares},
  Mathematical Surveys and Monographs, no. 146, American Mathematical Society,
  Rhode Island, 2008.

\bibitem[M{\"o}l76]{moller76}
H.~M. M{\"o}ller, \emph{{K}ubaturformeln mit minimaler {K}notenzahl}, Numer.\
  Math. \textbf{25} (1976), 185--200.

\bibitem[MVKW95]{milanf95}
P.~Milanfar, G.~Verghese, W.~Karl, and A.~Willsky, \emph{Reconstructing
  polygons from moments with connections to array processing}, IEEE Trans.\
  Signal Proc. \textbf{43} (1995), 432--443.

\bibitem[Nie14]{nie14}
J.~Nie, \emph{The $\mathcal{A}$-truncated ${K}$-moment problem}, Found.\ Comput.\
  Math. \textbf{14} (2014), 1243--1276.

\bibitem[Rez92]{reznick92}
B.~Reznick, \emph{Sums of even powers of real linear forms}, Mem. Amer. Math.
  Soc., vol.~96, American Mathematical Society, 1992, {MEMO/0463}.

\bibitem[Ric57]{richte57}
H.~Richter, \emph{Parameterfreie {A}bsch\"atzung und {R}ealisierung von
  {E}rwartungswerten}, Bl.\ Deutsch.\ Ges.\ Versicherungsmath. \textbf{3}
  (1957), 147--161.

\bibitem[Rog58]{rogosi58}
W.~W. Rogosinski, \emph{Moments of non-negative mass}, Proc.\ R.\ Soc.\ Lond.\
  A \textbf{245} (1958), 1--27.

\bibitem[Ros52]{rosenb52}
P.~C. Rosenbloom, \emph{Quelques classes de probl\`{e}me extr\'{e}maux. {II}},
  Bull.\ Soc.\ Math.\ France \textbf{80} (1952), 183--215.

\bibitem[RS18]{rienerOptima}
C.~Riener and M.~Schweighofer, \emph{Optimization approaches to quadrature: new
  characterizations of {G}aussian quadrature on the line and quadrature with
  few nodes on plane algebraic curves, on the plane and in higher dimensions},
  J.~Compl. \textbf{45} (2018), 22--54.

\bibitem[Sch00]{scheider00}
C.~Scheiderer, \emph{Sums of squares of regular functions on real algebraic
  varieties}, Trans. Amer. Math. Soc. \textbf{352} (2000), 1039--1069.

\bibitem[Sch17]{schmudMomentBook}
K.~Schm\"{u}dgen, \emph{The {M}oment {P}roblem}, Springer, New York, 2017.

\bibitem[Sha77]{shafa77}
I.~R. Shafarevich, \emph{Basic algebraic geometry}, Springer-Verlag, Berlin-New
  York, 1977, Translated from the Russian by K. A. Hirsch, Revised printing of
  Grundlehren der mathematischen Wissenschaften, Vol. 213, 1974.

\bibitem[ST43]{shohat43}
J.~A. Shohat and J.~D. Tamarkin, \emph{The {P}roblem of {M}oments}, Amer.\
  Math.\ Soc., Providence, R.I., 1943.

\bibitem[Sta78]{stanley78}
R.~P. Stanley, \emph{Hilbert functions of graded algebras}, Adv.\ Math.
  \textbf{28} (1978), 57--83.

\bibitem[Sti94]{stielt94}
T.~J. Stieltjes, \emph{Recherches sur les fractions continues}, Ann.~Fac.\ Sci.
  Toulouse \textbf{8} (1894), no.~4, J1--J122.

\bibitem[Str71]{stroud71}
A.~H. Stroud, \emph{Approximate {C}alculation of {M}ultiple {I}ntegrals},
  Prentice Hall, 1971.

\end{thebibliography}

\providecommand{\bysame}{\leavevmode\hbox to3em{\hrulefill}\thinspace}
\providecommand{\MR}{\relax\ifhmode\unskip\space\fi MR }
\providecommand{\MRhref}[2]{%
  \href{http://www.ams.org/mathscinet-getitem?mr=#1}{#2}
}
\providecommand{\href}[2]{#2}

\end{document}